\documentclass[12pt]{article}

\usepackage{amsmath}
\usepackage{amssymb}
\usepackage{amsthm}
\usepackage{bbm}
\usepackage{graphicx}
\usepackage{color}
\usepackage{epsfig}

\newtheorem{theorem}{Theorem}

\newtheorem{prop}[theorem]{Proposition}
\newtheorem{lemma}[theorem]{Lemma}

\newcommand{\N}{\mathbb{N}}

\newcommand{\R}{\mathbb{R}}

\allowdisplaybreaks

\title{Iterated Point-Line Configurations Grow Doubly-Exponentially}
\author{Joshua Cooper and Mark Walters}

\begin{document}

\maketitle
\begin{abstract}
Begin with a set of four points in the real plane in general position.  Add to this collection the intersection of all lines through pairs of these points.  Iterate.  Ismailescu and Radoi\v{c}i\'{c} (2003) showed that the limiting set is dense in the plane.  We give doubly exponential upper and lower bounds on the number of points at each stage.  The proof employs a variant of the Szemer\'edi-Trotter Theorem and an analysis of the ``minimum degree'' of the growing configuration.
\end{abstract}

Consider the iterative process of constructing points and lines in the real plane given by the following:  begin with a set of points $P_1 = \{p_1,p_2,p_3,p_4\}$ in the real plane in general position.  For each pair of points, construct the line passing through the pair.  This will create a set of lines $L_1 = \{\ell_1,\ell_2,\ell_3,\ell_4,\ell_5,\ell_6\}$.  Some of these constructed lines will intersect at points in the plane that do not belong to the set $P_1$.  Add any such point to the set $P_1$ to get a new set $P_2$.  Now, note that there exist some pairs of points in $P_2$ that do not lie on a line in $L_1$, namely some elements of $P_2 \setminus P_1$.  Add these missing lines to the set $L_1$ to get a new set $L_2$.  Iterate in this manner, adding points to $P_k$ followed by adding lines to $L_k$.  We assume that the original configuration is such that for every $k \in \N$ no two lines in $L_k$ are parallel.

Now we introduce some notation for this iterative process.  The $k^\textrm{th}$ \textit{stage} is defined to consist of these two ordered steps:
\begin{enumerate}
\item Add each intersection of pairs of elements of $L_k$ to $P_{k+1}$, and
\item Add a line through each of pair of elements of $P_k$ to $L_{k+1}$.
\end{enumerate}

Under this definition, we say that stage 1 begins with the configuration of four points with six lines and stage $k$ begins with $n_k$ points with $m_k$ lines.  We will denote the set of points at the beginning of stage $k$ by $P_k$ and likewise the set of lines at the beginning of stage $k$ by $L_k$. There are some trivial bounds on the number of points and lines at stage $k$ that can be obtained with this notation.  Since a point in $P_k$ must lie at the intersection of at least two lines of $L_{k-1}$ we know that at stage $k$, there are at most $\binom{m_{k-1}}{2}$ points.  Similarly, since a line in $L_k$ must contain at least two points from $P_k$ we know that at stage $k$ there are at most $\binom{n_k}{2}$ lines.  In other words, $$n_k \le \binom{m_{k-1}}{2} \hskip24pt \hbox{and} \hskip24pt m_k \le \binom{n_k}{2}.$$  From this it follows that $$n_{k+1} \le \binom{m_k}{2} \le \binom{\binom{n_k}{2}}{2} < \binom{\frac{{n_k}^2}{2}}{2} < \frac{\left(\frac{{n_k}^2}{2}\right)^2}{2} = \frac{{n_k}^4}{8}$$ and $$m_{k+1} \le \binom{n_{k+1}}{2} \le \binom{\binom{m_k}{2}}{2} < \binom{\frac{{m_k}^2}{2}}{2} < \frac{\left(\frac{{m_k}^2}{2}\right)^2}{2} = \frac{{m_k}^4}{8}.$$

Note that a stage in this iterative process can be alternatively defined as follows:
\begin{enumerate}
\item Place a point at any intersection of a pair of lines for which a point does not already exist.
\item Take the dual of the configuration of points and lines (points become lines and lines become points).
\item Return to step 1.
\end{enumerate}
Hence, points and lines play a very similar role in this process and we only need to consider bounds on one of the two quantities.  Henceforth we will only provide arguments concerning the bounds on $n_k$. A trivial lower bound is given in the following:

\begin{prop} For all $k \in \N$, $n_{k+1} \geq n_k + 1$.
\end{prop}

\begin{proof} If this claim is false then we must have a stage at which the process stabilizes \cite{IR03}.  So, suppose that the process stabilizes at the beginning of stage $k$ and let conv($P_k$) denote the convex hull of $P_k$, where $|$conv($P_k)|$ denotes the number of vertices of this convex hull.  Suppose first that $|$conv($P_k)| \ge 4$.  In this case, we can find two nonadjacent, nonparallel sides of the convex hull, which lie on lines that intersect outside of the convex hull.  This contradicts the stability supposition.  So, $|$conv($P_k)| = 3$. Let $\{a,b,c\}$ be the set of vertices of the triangle forming the convex hull. Suppose that there exist points along at least two of the sides of the triangle defined by $\{a,b,c\}$, say $x \in ab$ and $y \in bc$.  In this case, the line formed by $xy$ must intersect $ac$ outside the convex hull, again contradicting stability.  So,
there exist points along at most one of the sides of the triangle defined by $\{a,b,c\}$.  Suppose that there exists some point $x$ in the interior of $\{a,b,c\}$ and define $y = ax \cap bc$, $z = cx \cap ab$.  In this case, we have $y \in bc$ and $z \in ab$, a contradiction to the assumption that at most one side of the triangle contains points. The only remaining possibility is that $P_k$ is comprised of $n_k - 1$ collinear points.  But, the starting configuration of points and lines has the condition that for any line in $L_1$, there are at least two points of $P_1$ not passing through it.  Since we never remove any points during this process, then this must hold true for every stage, in particular stage $k$.  This contradiction completes the proof.
\end{proof}

We define the \textit{degree} of a point $p \in P_k$, denoted $d_k(p)$, to be the number of distinct lines incident upon $p$ at the beginning of stage $k$.  Similarly, the degree of a line $\ell \in L_k$, denoted $d_k(\ell)$, is the number of distinct points through which it passes at the beginning of stage $k$.  Also, let $$\delta_k = \min \{d_k(p) \mid p \in P_k \} \hskip.5in \hbox{and} \hskip.5in \overline{\delta}_k = \min \{d_k(\ell) \mid \ell \in L_k \}$$ and $$\Delta_k = \max \{d_k(p) \mid p \in P_k \} \hskip.5in \hbox{and} \hskip.5in \overline{\Delta}_k = \max \{d_k(\ell) \mid \ell \in L_k \}.$$  Define an \textit{$n \times n$ grid} to be any configuration of two collections of $n$ parallel lines, where the one collection is not parallel to the other.  Using these definitions, we obtain the following observation:

\begin{prop} For all $k \in \N$, $\delta_k \geq 3$.
\end{prop}
\begin{proof}
Suppose to the contrary that there exists some $k \in \N$ with $\delta_k < 3$. Since there are no points of degree 1, we must have $\delta_k = 2$.  So there exists $p \in P_k$ with $d_k(p) = 2$, i.e., there exist two lines $\ell,\ell^{\prime} \in L_k$ with $P_k \subseteq \ell \cup \ell^{\prime}$.  Note that $n_2 = 7$ and $\overline{\Delta}_2 = 3$ and so for all $\ell,\ell^{\prime} \in L_2$, $P_2 \nsubseteq \ell \cup \ell^{\prime}.$  Since we never remove points in this iterative process, we know that if there exists $\ell,\ell^{\prime} \in L_k$ with $P_k \subseteq \ell \cup \ell^{\prime}$, then $k<2$, i.e., $k=1$.  But we know that $\delta_1 = 3,$ a contradiction.
\end{proof}

We can obtain major improvements to the trivial lower bound using the following:

\begin{lemma} The minimum number of parallel lines required to pass through all of the intersections of an $n \times n$ grid is $2n-1$.
\end{lemma}
\begin{proof}
Suppose that $Q$ and $R$ are sets of parallel lines that comprise an $n \times n$ grid.  Let $S$ be a minimal witness set of $s$ parallel lines passing through all intersections of the grid. We aim to show that $s \ge 2n-1$. Without loss of generality, orient the grid so that the lines of $S$ are vertical in the $xy$-plane and let $X = \{x_1,x_2,\dots,x_s\}$ be the $x$-intercepts of the lines of $S$. So $X$ is the collection of projected points, when we project the grid intersections onto the $x$-axis with this orientation. Let $\pi(p)$ denote the projection of a point $p$ in the grid onto the $x$-axis. Arbitrarily choose lines $\ell_q,\ell_r$ in the grid with $\ell_q \in Q$ and $\ell_r \in R$. Let $q_1,q_2,\dots,q_n$ and $r_1,r_2,\dots,r_n$ be the points of intersection of $\ell_q$ with $R$ and $\ell_r$ with $Q$, respectively, where $$\pi(q_1) \le \pi(q_2) \le \cdots \le \pi(q_n)$$ and 
$$
\pi(r_1) \le \pi(r_2) \le \cdots \le \pi(r_n).
$$ 
Suppose also that $q_i = r_j$. Define $A$ and $B$ to be the sets of real numbers given by
$$
A=\{\pi(q_1),\pi(q_2),\dots,\pi(q_n)\}$$ and $$B=\{\pi(r_1)-\pi(r_j),\pi(r_2)-\pi(r_j),\dots,\pi(r_n)-\pi(r_j)\}.
$$ 
Under this setting we have that $S = A+B$ and thus $$s = |A+B|.$$  It is well known that $$|A+B| \ge 2n-1$$ for any pair $A,B$ of sets of cardinality $n$ and that equality is achieved when $A$ and $B$ are arithmetic progressions \cite{N96}.  It follows that $s \ge 2n-1$, completing the proof.
\end{proof}

Using this lemma we can prove the following:

\begin{theorem} $\delta_{k+1} \geq \min \{n_k-1, 2 \delta_k -3\}.$
\end{theorem}
\begin{proof} Let $p \in P_k$. It suffices to show that $$d_{k+1}(p) \ge \min \{n_k-1, 2 \delta_k -3\}.$$ First suppose each line in $L_k$ that passes through $p$ has degree 2. In this case, it's easy to see that there are $d_k(p)+1$ points at the beginning of stage $k$ and so $d_k(p) = n_k-1$.  Since we never remove lines, we know that
\begin{align*}
d_{k+1}(p) &\ge d_k(p) \\
&= n_k-1 \\
&\ge \min \{n_k - 1, 2 \delta_k - 3 \}.
\end{align*}
 Now suppose there exists a line $\ell \in L_k$ that passes through $p$ with $d_k(\ell) \ge 3$. Let $q,r \in P_k$ be the other two points on $\ell$. Note that $d_k(q) \ge \delta_k$ and $d_k(r) \ge \delta_k$ and so there exist two sets of lines $$L_q = \{\ell_{q_1},\ell_{q_2},\dots,\ell_{q_n}\} \subseteq L_k \backslash \ell \hskip.25in \hbox{and} \hskip.25in L_r = \{\ell_{r_1},\ell_{r_2},\dots,\ell_{r_m}\} \subseteq L_k \backslash \ell,$$ where $n,m \ge \delta_k-1$ and the sets $L_q \cup \ell$ and $L_r \cup \ell$ consist of the lines incident upon $q$ and $r$, respectively. Now, consider the real plane as a subset of the real projective plane in the standard way and let $\ell$ be the line at infinity.  We restrict our attention to arbitrarily chosen subsets ${L_q}^{\prime} \subseteq L_q$ and ${L_r}^{\prime} \subseteq L_r$, where $|{L_q}^{\prime}| = |{L_r}^{\prime}| = \delta_k - 1.$  These lines form a $(\delta_k - 1) \times (\delta_k - 1)$ grid. Now in this grid we will place a point at each intersection for which one does not already exist during stage $k$. After doing so, we will construct a line through each pair of points for which one does not already exist. In particular, we will do so for pairs of points of the form ($p,x$), where $x$ lies at the intersection of lines from ${L_q}^{\prime}$ and ${L_r}^{\prime}$. So, at the beginning of stage $k+1$, there will be at least $s$ lines incident upon $p$, where $s$ denotes the number of lines necessary to adjoin $p$ with all of the intersections of the grid. In other words, $d_{k+1}(p) \ge s$. Note that any lines passing through $p$ would form a third collection of parallel lines to add to the grid. Therefore, $s$ is at least the minimum number of parallel lines required to pass through all of the intersections of a $(\delta_k - 1) \times (\delta_k - 1)$ grid. Applying Lemma 1 yields
\begin{align*}
d_{k+1}(p) &\ge s \\
&\ge 2(\delta_k - 1) - 1 \\
&= 2 \delta_k - 3 \\
&\ge \min \{n_k-1, 2 \delta_k -3\}.
\end{align*}
\end{proof}

Now by using techniques similar to the preceding proofs, we can obtain even faster growth of the minimum degree.  We will then use the growth rate of $\delta_k$ to provide arguments for a better lower bound on $n_k$.  First, let $cr(G)$ denote the crossing number of a graph, which is the minimum number of crossings in a planar drawing of the graph $G$.  We will use the following lemma regarding crossing numbers (the proof can be found in \cite{PaTo97}):

\begin{lemma} If a graph $G$ with $n$ vertices and $e$ edges has $e > 7.5n$, then we have $$cr(G) \ge \frac{e^3}{33.75n^2}.$$
\end{lemma}

\noindent We now use this crossing number inequality in the following theorem.  The argument closely resembles Sz\'ekely's proof (\cite{S97}) of the Szemer\'edi-Trotter Theorem (first appearing in \cite{SzTr83}).

\begin{theorem} Let $\mathcal{F} = \{F_1,F_2,\dots,F_N\}$ be a collection of $N \ge 4$ families, each of exactly $k \ge 2$ parallel lines, no two collections parallel to each other. Let $P$ denote the collection of points that lie at the intersections of lines $\ell_i$ and $\ell_j$, where $\ell_i \in F_1$ and $\ell_j \in F_j$ for some $2 \le j \le k$. Then $$|P| \ge ck^2N^{1/2},$$ where $c$ is a positive real constant.
\end{theorem}

\begin{proof} Let $A$ denote this configuration of $|P|$ points and $Nk$ lines. Let $i$ be the number of point-line incidences in $A$. Note that there are $N$ different families of parallel lines in $A$, each containing exactly $k$ lines. For all families except $F_1$, each line contains exactly $k$ points from $P$ and thus contains exactly $k-1$ line segments which connect two points from $A$, call them edges.  We know that $k \ge 2$ and so $k-1 \ge k/2$.  Hence, each line contains at least $k/2$ edges and if we add this up over all of the $Nk$ lines, we see that the number of edges obtained in this manner is at least half of the total number of incidences.  In other words, $$\hbox{(total number of edges)} \ge \frac{i}{2}.$$  Now, we can count the exact number of edges in $A$.  For the $k$ lines of $F_1$, there are $|P| - k$ edges because all $|P|$ points lie on the lines of $F_1$ and for each line we must subtract one to count the number of edges.  For each of the remaining $N-1$ families there are exactly $k$ lines, each containing exactly $k-1$ edges, yielding a total of $$(N-1)k(k-1)$$ edges. Adding these quantities together, we obtain a grand total of $$|P| - k + (N-1)k(k-1)$$ edges, which simplifies to $$|P| + Nk(k-1) - k^2.$$

Now consider the graph $G$ with $V(G) = P$ and $E(G)$ consisting of the aforementioned edges.  Since all of the edges lie on one of $Nk$ lines, and any two lines intersect in at most one point, we have $$\hbox{cr}(G) \le (Nk)^2.$$  Applying the crossing number inequality, we obtain that either
\begin{equation}
|P| + Nk(k-1) - k^2 \le 7.5|P| \label{1}
\end{equation}
 or that
\begin{equation}
(Nk)^2 \ge \frac{(|P| + Nk(k-1) - k^2)^3}{33.75|P|^2} \label{2}
\end{equation}

In the case of (1) we get $$Nk(k-1) - k^2 \le 6.5|P|$$ which implies that $$\frac{Nk(k-1) - k^2}{6.5} \le |P|.$$ Now, since we know that $k \ge 2$ and $N \ge 4$, we have $k-1 \ge k/2$ and $N-2 \ge N^{1/2}$.  Combining this with the previous equation yields $$|P| \ge \frac{Nk(k-1) - k^2}{6.5} \ge \frac{(N-2)k^2}{13} \ge c_1k^2N^{1/2}$$ for some positive constant $c_1$.

In the case of (2) we have $$33.75 |P|^2 (Nk)^2 \ge (|P| + Nk(k-1) - k^2)^3$$ and so $$c_2|P|^{2/3} (Nk)^{2/3} \ge |P| + Nk(k-1) - k^2$$ for some positive constant $c_2$. Recall that the RHS of this inequality is $|E(G)|$, which is at least $i/2$. So we have $$\frac{i}{2} \le c_2|P|^{2/3}(Nk)^{2/3}.$$ Now, since each of the $Nk$ lines in $A$ must pass through at least $k$ points, then there are at least $Nk^2$ incidences. From this it follows that $$Nk^2 \le i \le c_3|P|^{2/3}(Nk)^{2/3}$$ for some positive constant $c_3$. Hence, $$N^3k^6 \le c_4|P|^2(Nk)^2$$ and so $$|P| \ge c_5k^2N^{1/2}$$ for some positive constants $c_4$ and $c_5$. So in both cases, we end up with our desired result.

\end{proof}

Now, we can use the previous result to prove the following lemma regarding degree growth:

\begin{lemma} Given any point $p \in P_k$ with $d_k(p)=d$, there exists a positive real constant $c$ such that
$$
d_{k+1}(p) \geq c \delta_k \left( \frac{n_k}{d} \right)^{1/2}
$$
\end{lemma}
\begin{proof}
Let $p \in P_k$ with $d_k(p) = d$. By the pigeonhole principle, there exists some line $\ell$ through $p$ with at least $s = \frac{n_k-1}{d}$ points on it (excluding $p$). Since each of these $s$ points has at least the minimum degree, we know that there are at least $\delta_k - 1$ lines through each point (excluding $\ell$).  Consider the real plane as a subset of the real projective plane in the standard way and let $\ell$ be the line at infinity. If we restrict our attention to only the points on $\ell$ and the lines through them, then we obtain a grid of $s+1$ families of parallel lines, one family for each of the points on $\ell$.  Each family of parallel lines contains at least $\delta_k - 1$ lines and no two families can be parallel (since they come from distinct points). We would like to restrict our attention to families of exactly $\delta_k -1$ parallel lines.  So for each family, except for the one generated by $p$, arbitrarily choose a subset of $\delta_k-1$ lines and disregard all other lines in that family. Let $F$ be the family of lines through $p$ and choose one family $R$ to be a set of ``reference'' lines. Let $P_0$ denote the set of points that lie at the intersection of a reference line and one of the other $s-1$ families (excluding $F$).

Now, during stage $k$, a point must be added to any intersection for which one does not already exist, in particular all points of $P_0$. Also, a line must be added to connect any pair of points for which one does not already exist, in particular for the pairs in the set $T = \{(p,q) \mid q \in P_0\}$. Let $t$ denote the number of distinct lines generated by pairs in the set $T$.  Note that any such line can pass through at most $\delta_k -1$ points of $P_0$ because all the points of $P_0$ lie in the family $R$, which contains exactly $\delta_k -1$ lines.  It follows that
\begin{equation}
d_{k+1}(p) \ge t \ge \frac{|P_0|}{\delta_k-1} \ge \frac{|P_0|}{\delta_k}, \label{3}
\end{equation}
with the first inequality holding because any line generated by the set $T$ must pass through $p$, and hence contributes to its degree in the stage.

Now, for the moment, exclude $F$ from our collection of families and consider all other families of lines along with the points of $P_0$. Suppose $s < 4$.  Hence, $n_k-1 < 4d$ and so $n_k < 4d+1$, i.e., $n_k \le 4d$. It follows that $$\left(\frac{n_k}{d}\right)^{1/2} \le 2.$$  Note that $d_{k+1}(p) \ge \delta_k$ for all $p \in P_{k+1}$ and so $d_{k+1}(p) \ge 2c \delta_k$ holds true for $c = \frac{1}{2}$.  Hence, $$d_{k+1}(p) \ge c \delta_k \left(\frac{n_k}{d}\right)^{1/2}$$ for some positive real constant $c$, as desired.  Now suppose that $s \ge 4$.  Since we also know that $\delta_k \ge 3$, i.e., $\delta_k-1 \ge 2$ for all $k \in \N$, we can apply Theorem 2 to this configuration with $F_1 = R$, $N = s$, and $k = \delta_k - 1$.  It follows that $$|P_0| \ge c_1(\delta_k - 1)^2\left(\frac{n_k-1}{d}\right)^{1/2}$$ for some positive constant $c_1$. Now, note that $\delta_k \ge 2$ and $n_k \ge 4$, which implies that $\delta_k - 1 \ge \frac{1}{2}\delta_k$ and $n_k - 1 \ge \frac{3}{4}n_k$. Combining this with the previous equation, we obtain
\begin{equation}
|P_0| \ge c_1\left(\frac{\delta_k}{2}\right)^2 \left(\frac{3n_k}{4d} \right)^{1/2} \ge c_2 {\delta_k}^2 \left(\frac{n_k}{d}\right)^{1/2} \label{4}
\end{equation}
for some positive constant $c_2$. If we combine (3) and (4) we get
\begin{equation}
d_{k+1}(p) \ge c_2 \delta_k \left(\frac{n_k}{d}\right)^{1/2}, \label{5}
\end{equation}
as desired.  This completes the proof.
\end{proof}

Note that $|P_0| \le n_{k+1}$ and so $$n_{k+1} \ge c {\delta_k}^2 \left( \frac{n_k}{d} \right)^{1/2}$$ must hold for any point $p \in P_k$ with $d_k(p)=d$ and some positive constant $c$.  In particular, it must hold for $p \in P_k$ chosen with $d_k(p) = \delta_k$. In this case
\begin{equation}
n_{k+1} \ge c {\delta_k}^2 \left(\frac{n_k} {\delta_k}\right)^{1/2} = c {\delta_k}^{3/2}{n_k}^{1/2}. \label{6}
\end{equation}
Now we are able to provide an improved lower bound on the minimum degree, which will be used to improve the lower bound on $n_k$.

\begin{lemma} Given any $k \in \N$, $\epsilon \ge 0$, and any positive real constant $c_1$ such that $\delta_k \ge c_1{n_k}^{\epsilon}$, there exists some positive real constant $c_2$ such that $$\delta_{k+1} \geq c_2 {n_k}^{\left(\frac{1 + 2 \epsilon}{3}\right)}.$$
\end{lemma}
\begin{proof}
Suppose that $\delta_k \ge c_1{n_k}^{\epsilon}$ for some $k \in \N$, $\epsilon \ge 0$, and positive real constant $c_1$.  Define $\alpha \in \R$ by $$\alpha = \frac{1 + 2 \epsilon}{3}.$$ Let $p \in P_k$ with $d_k(p) = d$. There are two cases:  either $d < {n_k}^{\alpha}$ or $d \ge {n_k}^{\alpha}$.  If $d < {n_k}^{\alpha}$, then by Lemma 7 we have $$d_{k+1}(p) \ge c_0 \delta_k \left(\frac{n_k}{{n_k}^{\alpha}}\right)^{1/2} = c_0 \delta_k {n_k}^{\frac{1-\alpha}{2}}$$ for some positive real constant $c_0$.  Since $\delta_k \ge c_1{n_k}^{\epsilon}$ and $\alpha = (1 + 2 \epsilon)/3$, we must have
\begin{align*}
d_{k+1}(p) &\ge c_0 \delta_k {n_k}^{\frac{1-\alpha}{2}} \\
&\ge c_0 c_1 {n_k}^{\epsilon}{n_k}^{\frac{1-\epsilon}{3}} \\
&= c_2 {n_k}^{\left(\frac{1 + 2 \epsilon}{3}\right)}, \\
\end{align*}
where $c_2 = c_0 c_1$.
If instead $d \ge {n_k}^{\alpha}$, then obviously we have $$d_{k+1}(p) \ge d \ge {n_k}^{\alpha} \ge c_2{n_k}^{\left(\frac{1 + 2 \epsilon}{3}\right)},$$ where $c_2 \le 1$.  So, in both cases, we have the conclusion that $$d_{k+1}(p) \ge c_2{n_k}^{\left(\frac{1 + 2 \epsilon}{3}\right)}$$ and this will hold true for any $p \in P_k$. Since the choice of $p \in P_k$ was arbitrary, we have
\begin{equation}
\delta_{k+1} \geq c_2{n_k}^{\left(\frac{1 + 2 \epsilon}{3}\right)}, \label{7}
\end{equation}
as desired.  This completes the proof.
\end{proof}

Now, we are able to obtain some numerical results from Lemma 8.  Note first that
\begin{equation}
\delta_k \ge c_2{n_{k-1}}^{1/3} \label{8}
\end{equation}
for some positive real constant $c_2$ (letting $\epsilon = 0$). Further recall that the trivial upper bound yields
\begin{equation}
n_{k-1} \ge (8n_k)^{1/4}. \label{9}
\end{equation}
Combining (8) and (9), we get that for all $k \in \N$ $$\delta_k \ge c_2{n_{k-1}}^{1/3} \ge c_2[(8n_k)^{1/4}]^{1/3} \ge c_3{n_k}^{1/12}$$ for some positive real constant $c_3$.  Now, we can apply Lemma 8 with $\epsilon = \frac{1}{12}$ for any $k \in \N$.  Since $$\frac{1 + 2(\frac{1}{12})}{3} = \frac{7}{18},$$ we get
\begin{equation}
\delta_k \ge c_4{n_{k-1}}^{7/18} \label{10}
\end{equation}
for some positive real constant $c_4$. Now, we combine (10) with (9), to obtain that for all $k \in \N$ $$\delta_k \ge c_4{n_{k-1}}^{7/18} \ge c_4[(8n_k)^{1/4}]^{7/18} \ge c_5{n_k}^{7/72}$$ for some positive real constant $c_5$.  This process can be iterated and the limiting value of $\epsilon > 0$ is found by setting $$\epsilon = \frac{1 + 2 \epsilon}{12}$$ which implies that $$\epsilon = 0.1 + o(1).$$  Now, using Lemma 8 ($\epsilon = 0.1 + o(1)$) with (6), we obtain
\begin{align*}
n_{k+1} &\ge c {\delta_k}^{3/2}{n_k}^{1/2} \\
&\ge c \left( c^{\prime} {n_{k-1}}^{\frac{1+2(0.1+o(1))}{3}} \right)^{3/2}n_k^{1/2} \\
&\ge c^{\prime \prime}{n_{k-1}}^{1.1 + o(1)}  \tag{11}
\end{align*}
for some positive constants $c$, $c^{\prime}$, and $c^{\prime \prime}$. Using (11) along with the trivial upper bound, we obtain the following theorem:

\begin{theorem}
Given $k \in \N$, there exists real positive constants $c_1$ and $c_2$ such that
\begin{equation*}
c_1 4^{1.0488^k} \le n_k \le c_2 4^{4^k}. \tag{12}
\end{equation*}
\end{theorem}

\begin{proof}
Note first that $n_1 = 4$ and $n_2 = 7$.  From repeated use of (11) we get that there exist real positive constants $a_1,a_2,a_3,a_4$ such that $$a_1 4^{(1.1+o(1))^k} \le n_{2k+1} \le a_2 4^{4^{2k+1}}$$ and $$a_3 7^{(1.1+o(1))^{k-1}} \le n_{2k} \le a_4 4^{4^{2k}}.$$
Taking square roots, it follows that there exist real positive constants $c_1$ and $c_2$ such that $$c_1 4^{1.0488^k} \le n_k \le c_2 4^{4^k},$$ as desired.
\end{proof}

Theorem 9 shows that the growth of $n_k$ is indeed doubly-exponential, as the easy upper bound suggests.  However, a considerable gap still remains between the exponents.  While we have no rigorous argument providing improvements of either bound, computational results and heuristic reasoning suggest that the actual growth rate of $n_k$ is closer to the stated upper bound.

%\nocite{}
\bibliographystyle{plain}
\bibliography{myrefs}

\begin{thebibliography}{1}

\bibitem{IR03}
Dan Ismailescu and Rado{\v{s}} Radoi{\v{c}}i{\'{c}}.
\newblock A dense planar point set from iterated line intersections.
\newblock {\em Computational Geometry}, 27:257--267, 2004.

\bibitem{N96}
Melvyn~B. Nathanson.
\newblock {\em Additive Number Theory: Inverse Problems and the Geometry of
  Sumsets}.
\newblock Springer, 1996.

\bibitem{PaTo97}
J{\'a}nos Pach and G{\'e}za T{\'o}th.
\newblock Graphs drawn with few crossings per edge.
\newblock {\em Combinatorica}, 17(3):427--439, 1997.

\bibitem{S97}
L{\'a}szl{\'o}~A. Sz{\'e}kely.
\newblock Crossing numbers and hard {E}rd{\H o}s problems in discrete geometry.
\newblock {\em Combin. Probab. Comput.}, 6(3):353--358, 1997.

\bibitem{SzTr83}
Endre Szemer{\'e}di and William~T. Trotter, Jr.
\newblock Extremal problems in discrete geometry.
\newblock {\em Combinatorica}, 3(3-4):381--392, 1983.

\end{thebibliography}
\end{document}